\documentclass[12pt, reqno]{amsart}

\usepackage{amsfonts,latexsym,amsthm,amssymb,amsmath,amscd,euscript,bm}
\usepackage[sc]{mathpazo}
\usepackage[margin = 2cm]{geometry}
\usepackage{enumitem}
\usepackage{hyperref}

\usepackage{color}
\usepackage{hyperref}
\usepackage{url}
\usepackage{breakurl}
\newcommand{\bburl}[1]{\textcolor{blue}{\url{#1}}}

\usepackage{amsmath,amsthm,amssymb,comment,fullpage}
\usepackage{braket}
\usepackage{mathtools}
\usepackage{multirow}
\usepackage{trfsigns}
\usepackage{setspace}

\usepackage{caption}
\usepackage{mathrsfs}
\usepackage{latexsym}
\usepackage[dvips]{graphics}
\usepackage{graphicx}
\usepackage{epsfig}
\usepackage{amsmath,amsfonts,amsthm,amssymb,amscd}
\input amssym.def
\input amssym.tex
\usepackage{color}
\usepackage{hyperref}
\usepackage{url}

\numberwithin{equation}{section}

\newtheorem{thm}{Theorem}[section]

\newtheorem{cor}[thm]{Corollary}
\newtheorem{lem}[thm]{Lemma}

\newtheorem{defi}[thm]{Definition}

\theoremstyle{plain}

\newtheorem{rem}[thm]{Remark}

\newcommand\be{\begin{equation}}
\newcommand\ee{\end{equation}}
\newcommand\bee{\begin{equation*}}
\newcommand\eee{\end{equation*}}
\newcommand\bea{\begin{eqnarray}}
\newcommand\eea{\end{eqnarray}}
\newcommand\beae{\begin{eqnarray*}}
\newcommand\eeae{\end{eqnarray*}}
\newcommand\bi{\begin{itemize}}
\newcommand\ei{\end{itemize}}
\newcommand\ben{\begin{enumerate}}
\newcommand\een{\end{enumerate}}
\newcommand\bc{\begin{center}}
\newcommand\ec{\end{center}}
\newcommand\ba{\begin{array}}
\newcommand\ea{\end{array}}

\newcommand\frakfamily{\usefont{U}{yfrak}{m}{n}}
\DeclareTextFontCommand{\textfrak}{\frakfamily}

\newcommand{\foh}{\frac{1}{2}}  

\newcommand{\hphi}{\widehat{\phi}}

\renewcommand{\d}{{\mathrm{d}}}

\newcommand{\supp}{\operatorname{supp}}

\newcommand{\<}{\left\langle}
\renewcommand{\>}{\right\rangle}

\title
{
	Bounding the Order of Vanishing of Cuspidal Newforms via the $n$\textsuperscript{th} Centered Moments
}

\author[Dutta]{Sohom Dutta}
\email{\textcolor{blue}{\href{mailto:sdutta611@student.fuhsd.org}{sdutta611@student.fuhsd.org}}}
\address{Lynbrook High School, San Jose, CA 95129}

\author[Miller]{Steven J. Miller}
\email{\textcolor{blue}{\href{mailto:sjm1@williams.edu}{sjm1@williams.edu}},  \textcolor{blue}{\href{Steven.Miller.MC.96@aya.yale.edu}{Steven.Miller.MC.96@aya.yale.edu}}}
\address{Department of Mathematics and Statistics, Williams College, Williamstown, MA 01267}

\date{\today}



\begin{document}


\maketitle

\thispagestyle{empty}

\begin{abstract}
Building on the work of Iwaniec, Luo and Sarnak, we use the $n$-level density to bound the probability of vanishing to order at least $r$ at the central point for families of cuspidal newforms of prime level $N \to \infty$, split by sign. There are three methods to improve bounds on the order of vanishing: optimizing the test functions, increasing the support, and increasing the $n$-level density studied. Previous work has determined the optimal test functions for the $1$ and $2$-level densities in certain support ranges, with the effectiveness of the bounds only marginally increasing by the optimized test functions over simpler ones, and thus this is not expected to be a productive avenue for further research. Similarly the support has been increased as far as possible, and further progress is shown to be related to delicate and difficult conjectures in number theory. Thus we concentrate on the third method, and study the higher centered moments (which are similar to the $n$-level densities but combinatorially easier). We find the level at each rank for which the bounds on the order of vanishing is the best, thus producing world-record bounds on the order of vanishing to rank at least $r$ for every $r > 2$ (for example, our bounds for vanishing to order at least 5 or at least 6 are less than half the previous bounds, a significant improvement). Additionally, we explicitly calculate the optimal test function for the $1$-level density from previous work and compare it to the naive test functions for higher levels. In doing so, we find that the optimal test function for certain levels are not the optimal for other levels, and some test functions may outperform others for some levels but not in others. Finally, we explicitly calculate the integrals needed to determine the bounds, doing so by transforming an $n$-dimensional integral to a $1$-dimensional integral and greatly reducing the computation cost in the process. 
\end{abstract}

\tableofcontents

\thispagestyle{empty}

\section{Introduction}

\subsection{Background}

Building on the work of Iwaniec, Luo and Sarnak \cite{ILS} we determine new world records for bounding the probability of $L$-functions of cuspidal newforms of prime level $N\to \infty$, split by sign, vanishing to order at least $r$ at the central point. We recall some standard definitions (see \cite{IK} for more details), review connections between $L$-functions and Random Matrix Theory, and then state our improved results.

\begin{defi}[Cuspidal Newforms]
Let $H^\star_k(N)$ denote the set of holomorphic cusp forms of weight $k$ that are newforms of level $N$. For every $f\in H^\star_k(N)$, we have a Fourier expansion
\begin{equation}
f(z)\ = \ \sum_{n=1}^\infty a_f(n) e(nz).
\end{equation}
We set $\lambda_f(n) =  a_f(n) n^{-(k-1)/2}$, and obtain the $L$-function associated to $f$ 
\begin{equation}
L(s,f)\ =\ \sum_{n=1}^\infty \lambda_f(n) n^{-s}.\label{eq: L-function}
\end{equation}
The completed $L$-function is
\begin{equation}\label{eq:completed_L_func}
\Lambda(s,f) \ :=\ \left(\frac{\sqrt{N}}{2\pi}\right)^s
\Gamma\left(s+\frac{k-1}{2}\right) L(s,f).\end{equation}
Since $\Lambda(s,f)$ satisfies the functional equation $\Lambda(s,f) \ = \
\epsilon_f \Lambda(1-s,f)$ with $\epsilon_f = \pm 1$, $H^\star_k(N)$ splits into two disjoint subsets, $H^+_k(N) = \{
f\in H^\star_k(N): \epsilon_f = +1\}$ and $H^-_k(N) = \{ f\in
H^\star_k(N): \epsilon_f = -1\}$.
The associated symmetry group of $H^\star_k(N)$ is Orthogonal \emph{(O)}, $H^+_k(N)$ is Special Orthogonal even \emph{SO(even)}, and $H^-_k(N)$ is Special Orthogonal odd \emph{SO(odd)}.
\end{defi}

\subsection{Random Matrix Theory and Statistics of $L$-functions}
The Riemann Hypothesis states that all zeroes of the Riemann zeta function are either at the negative even integers (the trivial zeros, as these numbers are well understood!) or complex numbers with real part $1/2$ (the nontrivial zeros); the Generalized Riemann Hypothesis (GRH) asserts this is true for all $L$-functions, in particular for the cuspidal newforms we study.

It turns out that the behavior of many different objects in mathematics and physics are the same, and this has led to fruitful conversations where one subject suggests problems and predicts answers in another. In the early 1900s, random matrix theory was used for applications in statistics and harmonic analysis. However, a major advance in the subject was made with the seminal work of Eugene Wigner, who noticed a remarkable connection between fluctuations in the position of compound nuclei resonances and statistics for the eigenvalues of random matrices.  As more researchers began exploring random matrix theory, they discovered connections between random matrix theory (specifically the distribution of eigenvalues of matrices) and number theory (the distribution of the non-trivial zeros). See \cite{BFMT-B, Ha} and the references therein for the history and many of the results.

The first statistics studied were the $n$-level correlations and the spacings between adjacent zeros; see \cite{Hej, Mon, Od1, Od2, RudnickSarnak}, where it was observed that the behavior of zeros far from the central point converged to a universal behavior, independent of the arithmetic of the $L$-function. This led to the quest to find a new statistic that was sensitive to the behavior near the central point, which by the Birch and Swinnerton-Dyer Conjecture was known to be an important point to study. 

Katz and Sarnak \cite{KatzSarnak, KatzSarnak2} introduced a new statistic, the $n$-level density, which has different values for different families of $L$-functions, and essentially only depends on the zeros near the central point. They found that as the level of the forms approach infinity, the statistics for the zeroes of families of $L$-functions can be modelled by eigenvalue statistics for one of the classical compact groups (unitary, orthogonal, symplectic). 

\begin{defi}[$n$-level Density]\label{nleveldensity}
The $n$-level density of an $L$-function $L(s,f)$ is defined as
\begin{equation}
    D_n (f; \phi) \ := \  \sum_{\substack{j_1, \cdots, j_n \\ j_i \neq \pm j_k}} \phi \left( \frac{\log c_f}{2\pi} \gamma_f^{(j_1)}, \dots, \frac{\log c_f}{2\pi} \gamma_f^{(j_n)} \right)	\label{def:density}
\end{equation}
for a \emph{test function} $\phi: \mathbb{R}^n \to \mathbb{R}$ where $c_f$ is the analytic conductor of $f$. For many applications we assume $\phi$ is a non-negative even Schwartz function (see \ref{Schwartz}) with compactly supported\footnote{A function $f$ is supported in $(\sigma, \sigma)$ if $f(x) = 0$ for all $x$ with $|x| > \sigma$.} Fourier transform and $\phi(0, \ldots, 0) > 0$. 
\end{defi}

There is now an extensive series of papers showing that the $n$-level densities of various families of $L$-functions match the random matrix theory predictions; see for example the introduction in \cite{C--} for a review of the literature. Our main result is to use the $n$-level densities to bound how often forms in the family $\mathcal{F}_N$ vanish to a given order (or more), with $\mathcal{F}_N$ being cuspidal newforms of level $N$ and some fixed weight $k$, with $N \to \infty$ through the primes\footnote{This latter condition is for technical reasons; with additional work (see \cite{BBDDM}) one may take $N$ tending to infinity through all integers. We use the results on the $n$-level densities of cuspidal newform families as inputs, and thus with standard but tedious and technical work, we could remove the prime restriction.}

\subsection{Test Functions}
To get the best bounds, we must choose a good test function to use in the $n$-level density. As shown later (see Theorem \ref{thm:onelevelboundpr}), we require the test function to be even, non-negative, Schwartz, and have a Fourier transform\footnote{We define the Fourier transform of $g$ by $\widehat{g}(y) = \int_{-\infty}^\infty g(x) e^{-2\pi i x y}$.} with finite support; we will later see how to pass from such functions to bounds on vanishing. 

To obtain the most information about what is happening at the central point, we want our test function to be concentrated there and rapidly decay, ideally a delta spike at the origin. Unfortunately, the closer our test function is to a delta spike, the larger the support is of its Fourier transform; this is the mathematical instance of the Heisenberg Uncertainty Principle: we cannot localize both a function and its Fourier transform. The reason this is an obstruction is that the $n$-level density requires us to compute certain weighted sums of the $L$-function coefficients times the Fourier transform of the test function evaluated at the logarithm of the primes, and we can only compute these sums if the support of the Fourier transform is suitably restricted. Thus the goal is to find test functions as close to the delta spike as we can, subject to being able to compute the resulting sums on the Fourier transform side.

Throughout this paper, the main test function we use is the naive test function (so named as it is easy to use and easy to guess is worth using).

\begin{defi}[Naive Test Function]
The naive test function is the Fourier test function pair 
\begin{equation}
    \phi_{\rm naive}(x) \ = \ \left(\frac{\sin(\pi v_nx)}{(\pi v_nx)}\right)^2 \ \ \ , \ \ \ \widehat{\phi}_{\rm naive}(y) = \frac{1}{v_n}\left(y-\frac{|y|}{v_n}\right)
\end{equation}
for $|y| \ < \ v_n$ where $v_n$ is the support.
\end{defi}

Previous work has improved on the naive test function by finding the optimal test function for some of the $n$-level densities for certain ranges of support (see, for example, \cite{BCDMZ, FreemanMiller, ILS}). This is one of three ways to improve results; however, the optimal function for the 1-level density leads to such a small improvement (and there are similarly small gains in the higher levels) that this avenue is not pursued in this paper. The second approach is to try to increase the support for the $n$-level density, but doing so requires resolving difficult combinatorics and technical sums. There has been a recent breakthrough here, though, in work by Cohen et. al. \cite{C--}, who increased the support for the $n$-level density from $1/(n-1)$ to $2/n$. 

We turn to the third method, high level densities. There has been some progress here; Li and Miller \cite{LiM} used the 4-level density to obtain better bounds than those from the first and second level densities, for sufficiently large vanishing at the central point. Unfortunately the higher level densities have a disadvantage -- while they give better bounds for much larger than expected vanishing at the central point, they give worse bounds for small vanishing.

This trade-off has never been quantified and analyzed till now, and is the main goal of this project. In particular, we prove for each $r$ what is the optimal level density to use (for a fixed test function, usually the naive one), to bound the probability of vanishing to order at least $r$ at the central point.

\subsection{$n$-th Centered Moments}

The first results are due to Iwaniec-Luo-Sarnak \cite{ILS}, who computed the 1-level density for families of cuspidal newforms split by sign with support up to $2$. This was extended by Hughes-Miller \cite{HM} to the $n$-level; while it was expected their results should hold up to $2/n$ for the support, there were combinatorial obstructions and their results are only valid if the support is at most $1/(n-1)$; these complications were recently resolved in \cite{C--}, and now the $n$-level is known up to $2/n$. 

The following theorem from \cite{C--} is used to generate bounds on the order of vanishing. 

\begin{thm}\label{2nmainfirst}[Theorem 1.1 from \cite{C--}.]
Let $n \geq 2$ and ${\rm supp}(\phi) \subset (-\frac{2}{n}, \frac{2}{n})$. Define
\begin{eqnarray}
    \sigma_{\phi}^2 & \ := \ & 2\int_{-\infty}^{\infty} |y|\widehat{\phi}(y)^2 dy \nonumber\\
    R(m, i; \phi) & \ := \ & 2^{m-1}(-1)^{m+1}\sum_{l=0}^{i-1} (-1)^l \binom{m}{l} \nonumber\\ & & \left(-\frac{1}{2}\phi^m(0)  +\int_{-\infty}^{\infty} \cdots \int_{-\infty}^{\infty} \widehat{\phi}(x_2) \cdots \widehat{\phi}(x_{l+1}) \right. \nonumber\\ & &  \left. \int_{-\infty}^{\infty} \phi^{m-l}(x_1)\frac{\sin(2\pi x_1(1+|x_2|+\cdots+|x_{l+1}|))}{2\pi x_1}dx_1 \cdots dx_{l+1} \right) \nonumber\\
    S(n, a, \phi) & \ := \ & \sum_{l=0}^{\lfloor\frac{a-1}{2}\rfloor} \frac{n!}{(n-2l)!l!}R(n-2l,a-2l,\phi)\left(\frac{\sigma_{\phi}^2}{2}\right)^l 
\end{eqnarray} and
\begin{equation}\label{integralvalue2nfirst}
    \lim_{\substack{N\to\infty \\ N \text{\rm prime}}} \<
\left(D(f;\phi) - \< D(f;\phi) \>_\pm \right)^{n}\>_\pm  \ = \ (n-1)!! \sigma_{\phi}^n 1_{n \ \rm even} \pm S(n, a; \phi).
\end{equation}
\end{thm}

\footnote{With extra work, we can move to $N$ through the square-free integers in the limit.}

\subsection{Main Results}\label{sec:mainresults}

We use the $n$-level densities to bound how often forms in the family $\mathcal{F}_N$ vanish with $\mathcal{F}_N$ being cuspidal newforms of level $N$ and some fixed weight $k$, with $N \to \infty$ through the primes. 

In order to generate better bounds on the order of vanishing, as remarked there are three main methods we can turn to: optimizing the test function, increasing the support, and using higher levels. We concentrate on the last.

Because work done to find the optimal test function produces marginal improvements, we turn to using higher levels using the recently improved results with support $2/n$, thus producing record bounds for the order of vanishing as shown in the following tables. For example, the previous record for vanishing to order 5 or more was .06580440 and for order 6 it was .00853841 from \cite{LiM}, while we show for vanishing to order 5 or more it is at most .020408300 while 6 or more .003346510; these are not slight improvements but a decrease by more than a factor of two! The tables below show our results for the various symmetry groups; cuspidal newforms with even functional equations are SO(even), while those with odd signs are SO(odd).

\begin{center}
    \begin{tabular}{ |p{3cm}||p{3cm}|p{3cm}|p{3cm}|  }
 \hline
 \multicolumn{3}{|c|}{Lowest Bounds for Each Rank for G=SO(even)} \\
 \hline
 Rank & Level Used &Bound\\
 \hline
 2 & 1  & $0.43231300$\\
 4 & 2 & $0.066666667$\\
 6 & 6 & $0.003346510$\\
 8 & 8 & $0.000579210$\\
 10 & 10 & $1.14380\times10^{-6}$\\
 12 & 12 & $1.85901\times10^{-8}$\\
 14 & 14 & $2.59310\times10^{-10}$\\
 16 & 16 & $3.09185\times10^{-12}$\\
 18 & 18 & $3.26332\times10^{-14}$\\
 20 & 20 & $3.08920\times10^{-16}$\\
 \hline
\end{tabular}
\end{center}

\begin{center}
    \begin{tabular}{ |p{3cm}||p{3cm}|p{3cm}|p{3cm}|  }
 \hline
 \multicolumn{3}{|c|}{Lowest Bounds for Each Rank for G=SO(odd)} \\
 \hline
 Rank & Level Used &Bound\\
 \hline
 1 & N/A & $1.0000000$\\
 3 & 2 & $0.111111111$\\
 5 & 2 & $0.020408300$\\
 7 & 6 & $0.000292790$\\
 9 & 8 & $7.65596\times10^{-6}$\\
 11 & 10 & $1.53302\times10^{-7}$\\
 13 & 12 & $2.50956\times10^{-9}$\\
 15 & 16 & $3.03362\times10^{-11}$\\
 17 & 18 & $3.10549\times10^{-13}$\\
 19 & 20 & $4.18402\times10^{-17}$\\
 \hline
\end{tabular}
\end{center}

We find that although using higher levels can create better bounds, increasing the levels to an arbitrarily large value does not lead to better bounds for small rank. Increasing the level eventually produces trivial bounds (such as the percent that vanish to a given level is at most a number greater than 100\%!). Thus there are only a finite number of levels that need to be checked to find which level creates the optimal bound for each order of vanishing. Additionally, we find that the optimal test functions for certain levels will not necessarily outperform other test functions for higher levels, as shown by the worse results produced by the $1$-level optimal test function for higher levels as opposed to the naive test function; this was a very surprising result, namely that a function which is superior for one $n$ can become inferior to a test function it beat for larger $n$.

In addition to calculating bounds, we show explicit calculations for the integrals needed to generate bounds. For support $2/n$, we are able to reduce an $n$-dimensional integral to a $1$-dimensional integral, thus significantly reducing the computation time needed to evaluate the integral (it is unfortunately not solvable in closed form, and requires numerical approximation, which we do via Simpson's method).

After reviewing some number theory and complex analysis preliminaries we continue to a review of work done by \cite{ILS} to generate bounds for the $1$-level density. We show the bounds obtained by \cite{ILS} by both the naive and optimal test function while explicitly enumerating the optimal Fourier pair for the $1$-level density. We then continue to results for the bounds with higher levels for support $1/(n-1)$ and extend this to support $2/n$. Lastly, we display plots and tables of the bounds of the order of vanishing based on the rank and level.

\section{Preliminaries}

We record some needed definitions and standard results; see \cite{SS} for more details and proofs.

\begin{defi}[Fourier Transform]
The Fourier transform of a function $\phi(x)$ is
\begin{equation}
    F(\phi(x)) \ := \ \widehat{\phi}(y) \ = \ \int_{-\infty}^{\infty} \phi(x)e^{-2\pi ixy}dx.
\end{equation}
\end{defi}

\begin{thm}[Plancherel Theorem]
For a Fourier pair $f(x)$ and $\widehat{f}(y)$ such that $E(t)$ is square-integrable,
\begin{equation}
    \int_{-\infty}^{\infty} |f(x)|^2 dt \ = \ \int_{-\infty}^{\infty} |\widehat{f}(y)|^2 dy.
\end{equation}
\end{thm}
\begin{rem}
As a result of the polarization identity \footnote{The polarization identity relates the inner product of two vectors with the norm of the two vectors.}, for square-integrable functions $f$ and $g$, the Plancherel Theorem can be re-expressed as 
\begin{equation}
    \int_{-\infty}^{\infty} f(x)\overline{g(x)}dx \ = \ \int_{-\infty}^{\infty} \widehat{f}(y)\overline{\widehat{g}(y)}dy.
\end{equation}
\end{rem}

Recall the convolution of two functions $f$ and $g$ is defined by \be (f \ast g)(x) \ = \ \int_{-\infty}^\infty f(t) g(x-t) dx. \ee

\begin{thm}[Convolution Theorem]\label{convolutionthm}
For functions $f(t)$ and $g(t)$, 
\begin{equation}
    \mathcal{F}[f \ast g] \ = \ \mathcal{F}[f] \mathcal{F}[g].
\end{equation}
\end{thm}

\begin{defi}[Schwartz space]\label{Schwartz}
The Schwartz space is the set of all functions such that $f$ is infinitely differentiable and the derivatives of $f$ decay faster than any polynomial. A function is Schwartz if it is a function in the Schwartz space.
\end{defi}

\begin{defi}[Even Indicator Function]
The Even Indicator function is defined to be 
\begin{equation}
    1_{\rm even}(n) \ := \ \begin{cases}
        1 & \text{{\rm if} } x \text{{\rm is\ even}}\\
        0 & \text{{\rm if} } x \text{{\rm is\ odd}}.
    \end{cases}
\end{equation}
\end{defi}

Since we encounter integrals later which we cannot evaluate in closed form, we use Simpson's Rule to approximate the integral using only function values at certain points while having a low error term. 

\begin{thm}[Simpson's Rule]
Let $[a, b]$ be an interval that is split into $s$ equal sub-intervals with $s$ even and $h \ = \ (b-a)/s$. Then, the integral $\int_{a}^b f(x) dx$ is approximated by
\begin{equation}
    \int_{a}^b f(x) dx \ \approx \ \frac{h}{3}\left[f(x_0) + 4\sum_{j=1}^{\frac{n}{2}} f(x_{2j-1}) + 2\sum_{j=1}^{\frac{n}{2}-1} f(x_{2j})+f(x_n) \right]
\end{equation}
where $x_j=a+jh$ for $j = 0,1,2,...,n$ and the difference between the integral and the sum, ${\rm Err}(f)$, is bounded by 
\begin{equation}
    |{\rm Err}(f)| \ \le \ \frac{h^4}{180}(b-a)\max_{\epsilon \in [a, b]} |f^{(4)}(\mathcal{E})|.
\end{equation}
\end{thm}

\begin{thm}[Fubini's Theorem]
If $\iint_{X \times Y} |f(x, y)|d(x,y) < \infty$, 
\begin{equation}
    \iint_{X \times Y} |f(x, y|d(x, y) \ = \ \int_{X}\left(\int_{Y} f(x, y) dy\right)dx \ = \ \int_{Y}\left(\int_{X} f(x, y) dx\right)dy.
\end{equation}
\end{thm}

\begin{defi}[Double Factorial]
For positive integer $n$, 
\begin{equation}
     n!! \ := \ \begin{cases}
        n(n-2)(n-4)\cdots4\cdot2 & \text{{\rm for\ $n$\ even}} \\
        n(n-2)(n-4)\cdots3\cdot1 & \text{{\rm for\ $n$\ odd}}.
    \end{cases}
\end{equation}

\end{defi}

\begin{defi}\label{fnr}
Define the family of forms in $\mathcal{F}_N$ that vanish to exact order $r$ to be $\mathcal{F}_{N, r}$.
\end{defi}

\begin{defi}\label{pFN} 
Define the percent of $f \in \mathcal{F}_N$ that vanish to exact order $r$ to be
\begin{equation}
p_{r}(\mathcal{F}_N) \ := \  \frac{|\mathcal{F}_{N,r}|}{|\mathcal{F}_N|}.
\end{equation}
\end{defi}

\begin{defi} For a fixed $r$, define $p_r(\mathcal{F})$ to be the limit of the percent of forms in $\mathcal{F_N}$ whose order of vanishing is $r$ as $N$ tends to infinity through the primes: 
\begin{equation}
    p_r(\mathcal{F}) \ := \ \lim_{N \to \infty} \frac{|\mathcal{F}_{N,r}|}{|\mathcal{F}_N|}.
\end{equation}
\end{defi}

\begin{defi} For a fixed $r$, define $q_r(\mathcal{F})$ to be the limit of the percent of forms in $\mathcal{F_N}$ that vanish to order at least $r$ as $N$ tends to infinity through the primes: 
\begin{equation}
    q_r(\mathcal{F}) \ = \ \sum_{i \geq r} p_{i}(\mathcal{F}).
\end{equation}
\end{defi}

\begin{defi}
The one-level density of an $L$-function $L(s, f)$ is 
\begin{equation}
    D(f, \phi)\ :=\ \sum_{\gamma_f} \phi\left(\frac{\gamma_f}{2\pi}\log{c_f}\right),
\end{equation}
where $c_f$ is the analytic conductor of $f$. The average or expectation of $D(f, \phi)$ over a family $\mathcal{F}_N$ is 
\begin{equation}
    E(\mathcal{F}_N, \phi) \ := \ \frac{1}{|\mathcal{F}_N|}\sum_{f \in \mathcal{F}_N} D(f, \phi).
\end{equation}
\end{defi}

\begin{defi}
The mean of the $1$-level density of a family is denoted $\mathcal{F}_N$ is $\mu(\phi, \mathcal{F}_N)$. 
\end{defi}

\section{Bounds on the Order of Vanishing From the 1-Level Density}

As we use the same functions and similar techniques as \cite{ILS}, We quickly summarize their work to obtain bounds on the order of vanishing using the $1$-level density, and then discuss the complications that arise in extending these arguments to the $n$-level densities.

\begin{thm}\label{thm:onelevelboundpr}[Bounds on the $1$-level density from \cite{ILS}]
Let $\phi$ be a non-negative, even Schwartz function with ${\rm supp}(\hphi) \subset (-\sigma, \sigma)$ for some finite $\sigma$. Let $\mathcal{G}$ be the group associated to the family $\{\mathcal{F}_N\}$ (i.e., Unitary, Symplectic, Orthogonal, \normalfont \textrm{SO(even)}, \textrm{SO(odd))}. \emph{Set}
\begin{equation}\label{gDef}
g_{\mathcal{F}}(\phi) \ := \ \int_{-\infty}^{\infty} \widehat{\phi}(y)\widehat{W}_{\mathcal{G}(\mathcal{F})}(y)dy.
\end{equation}

\emph{For a given $r$, as $N\to \infty$ the percent of forms in the family $\mathcal{F}_N$ that vanish to order exactly $r$ is bounded by}
\begin{equation}\label{pBound1}
p_{r} \ \le \ \frac{1}{r}(g_{\mathcal{F}}(\phi)).
\end{equation}
\end{thm}

\begin{cor}\label{specificBoundsP} Let $\phi$ be the naive test function with ${\rm supp}(\hphi) \subset (-v_n, v_n)$, Then
\normalfont
\begin{eqnarray}
p_{r}(\mathcal{F}) & \  \le \ & \frac{1}{r}\left(\frac{1}{v_n}+\frac{1}{2}\right) \ \ \textrm{ for G = O} \nonumber\\
p_{r}(\mathcal{F}) & \ \le \ & \left\{
\begin{array}{lr}
\frac{1}{r}\left(\frac{1}{v_n}+\frac{1}{2}\right) & \text{if } v_n \ \le \ 1\\
 \frac{1}{r}\left(\frac{2}{v_n}-\frac{1}{2v_n^2}+\epsilon\right) & \text{if } v_n \geq 1
\end{array}
\right\}\ \  \textrm{for G = SO(even)} \nonumber\\
p_{r}(\mathcal{F}) & \ \le \ & \left\{
\begin{array}{lr}
\frac{1}{r}\left(\frac{1}{v_n}+\frac{1}{2}+\epsilon\right) & \text{if } v_n \ \le \ 1\\
 \frac{1}{r}\left(1+\frac{1}{2v_n^2}\right) & \text{if } v_n \ \geq \ 1
\end{array}
\right\}  \ \ \textrm{for G = SO(odd)}.
\end{eqnarray}
\end{cor}

\begin{proof}[Proof of Corollary \ref{specificBoundsP}]
    To use Theorem \ref{thm:onelevelboundpr} we need to choose a good pair of functions, $\phi$ and $\hphi$. In \cite{ILS} the authors remark that a particularly good choice is the following (what we call the naive test function):
\begin{equation}\label{fpair}
\phi_v(x) \ = \ \left(\frac{\sin(\pi vx)}{\pi vx}\right)^2, \ \ \ \  \widehat{\phi}_v(y) \ = \ \frac{1}{v_n}\left(1-\frac{|y|}{v_n}\right)
\end{equation}
for $|y|<v$. 

To find the bounds on the order of vanishing, we first must find $g$ for specific cases. From \eqref{gDef}, we get
\begin{eqnarray}\label{gvals}
g_{\mathcal{F}}(v_n) & \ = \ & \frac{1}{v_n}+\frac{1}{2} \ \  \textrm{ for G = O} \nonumber\\
g_{\mathcal{F}}(v_n) & \ = \ & \left\{
\begin{array}{lr}
\frac{1}{v_n}+\frac{1}{2} & \text{if } v_n \le 1\\
 \frac{2}{v_n}-\frac{1}{2v_n^2} & \text{if } v_n \geq 1
\end{array}
\right\} \ \  \textrm{for G = SO(even)} \nonumber\\
g_{\mathcal{F}}(v_n) & \ = \ & \left\{
\begin{array}{lr}
\frac{1}{v_n}+\frac{1}{2}, & \text{if } v_n \le 1\\
 1+\frac{1}{2v_n^2} & \text{if } v_n \geq 1
\end{array}
\right\} \ \ \textrm{for G = SO(odd)}.
\end{eqnarray}

We now want to find the bounds for the order of vanishing as a function of the support $v$. Substituting in our values for $g_{\mathcal{F}}(v)$ from \eqref{gvals} into \eqref{pBound1} we get
\begin{eqnarray}\label{specialPBoundsfirst}
p_{r}(\mathcal{F}) & \  \le \ & \frac{1}{r}\left(\frac{1}{v_n}+\frac{1}{2}\right) \ \ \textrm{ for G = O} \nonumber\\
p_{r}(\mathcal{F}) & \ \le \ & \left\{
\begin{array}{lr}
\frac{1}{r}\left(\frac{1}{v_n}+\frac{1}{2}\right) & \text{if } v \ \le \ 1\\
 \frac{1}{r}\left(\frac{2}{v_n}-\frac{1}{2v_n^2}\right) & \text{if } v \geq 1
\end{array}
\right\}\ \  \textrm{for G = SO(even)} \nonumber\\
p_{r}(\mathcal{F}) & \ \le \ & \left\{
\begin{array}{lr}
\frac{1}{r}\left(\frac{1}{v_n}+\frac{1}{2}\right) & \text{if } v \ \le \ 1\\
 \frac{1}{r}\left(1+\frac{1}{2v_n^2}\right) & \text{if } v_n \ \geq \ 1
\end{array}
\right\}  \ \ \textrm{for G = SO(odd)}.
\end{eqnarray}
\end{proof}


The best results to date give the support for the $1$-level density to be ${\rm supp}(\hphi) \subset (-2, 2)$. As a result, we get
\begin{eqnarray}\label{specialPBounds}
p_{r}(\mathcal{F}) & \  \le \ & \frac{1}{r} \ \ \textrm{ for G = O} \nonumber\\
p_{r}(\mathcal{F}) & \ \le \ & 
 \frac{7}{8r}
\ \  \textrm{for G = SO(even)} \nonumber\\
p_{r}(\mathcal{F}) & \ \le \ & 
 \frac{9}{8r}
  \ \ \textrm{for G = SO(odd)}.
\end{eqnarray}

The test function $\phi_v(x)$ is not the optimal test function that we can use. The function that satisfies 
\begin{equation}
    \widehat{\phi}_{\rm optimal}(y) \ := \ (f_0 \ast \overline{f_0})(y)
\end{equation}
for 
\begin{equation}
    f_0(x) \ := \ \frac{\cos\left(\frac{|x|}{2}-\frac{\pi+1}{4}\right)}{\sqrt{2}\sin\left(\frac{1}{4}\right)+\sin\left(\frac{\pi+1}{4}\right)} \ \ , \ \ 0 \le |x| \le 1
\end{equation}
for $G={\rm SO(even)}$ and 
\begin{equation}
    f_0(x) \ := \ \frac{\cos\left(\frac{|x|}{2}+\frac{\pi-1}{4}\right)}{3\sin\left(\frac{\pi+1}{4}\right)-2\sin\left(\frac{\pi-1}{4}\right)} \ \ , \ \ 0 < |x| < 1
\end{equation}
for $G={\rm SO(odd)}$ from \cite{ILS} has been proven to be the optimal test function to use for the $1$-level density and can create better bounds. We now want to calculate the test function using the values provided by \cite{ILS} and compare the bounds from the optimal test function to that of the naive test function. 

\begin{lem}
Let 
\begin{equation}
    f(x) \ := \ \frac{\cos\left(\frac{x}{2}-a\right)}{b} \ \ , \ \ 0 \le |x| \le 1.
\end{equation}
\begin{eqnarray}
    (f \ast \overline{f})(y) & \ = \ & \frac{1}{b^2}\left[\frac{1}{2}(|y|-1)\cos\left(\frac{|y|}{2}\right)+
    \cos\left(\frac{1}{2}-2a\right)\sin\left(\frac{1-|y|}{2}\right) \right. \nonumber\\ & & \ + \left. \frac{1}{2}|y|\cos\left(2a-\frac{|y|}{2}\right)+\sin\left(\frac{|y|}{2}\right)\right]
\end{eqnarray}
for $0 \le |y| \le 1$ and 
\begin{equation}
    (f \ast \overline{f})(y) \ = \ \frac{1}{b^2}\left[\frac{1}{2}(|y|-2)\sin\left(\frac{1-|y|}{2}\right)+\sin\left(\frac{1-|y|}{2}\right)\right]
\end{equation}
for $1 \le |y| \le 2$.
\end{lem}

\begin{proof}
Since $f$ is a real-valued function, $\overline{f} \ = \ f$. Therefore, we want to find $(f \ast f)(y)$. Since our resulting function is an even function with support $v_n=2$, we consider the values of $y \in [0,2]$ and reflect over the $y$-axis to find the values of $y \in [-2,0]$. Using the definition of a convolution, we get
\begin{equation}
    (f \ast f)(y) \ = \ \frac{1}{b^2}\int_{y-1}^{1} \cos\left(\frac{|x|}{2}-A\right)\cos\left(\frac{|y-x|}{2}-A\right) dx.
\end{equation}
We first consider the case of $y \in [0,1]$. We proceed to split the integral into 3 cases: $x<0$, $x>0$ and $y \geq x$, $x>0$ and $y < x$. The contribution to the integral from each case is $V_1$, $V_2$, and $V_3$ respectively. 

For $x < 0$, we get
\begin{align}
    V_1 \ &= \ \frac{1}{b^2}\int_{y-1}^{0} \cos\left(\frac{y}{2}\right) \cos\left(x+2a-\frac{y}{2}\right)dx \nonumber\\ \ &= \ \frac{1}{b^2}\left[\frac{1}{2}(y-1)\cos\left(\frac{y}{2}\right)+\cos\left(\frac{1}{2}-2a\right)\sin\left(\frac{1-y}{2}\right)\right].
\end{align}

From $x > 0$ and $y \geq x$, we get
\begin{align}
    V_2 \ &= \ \frac{1}{b^2}\int_{0}^{y} \cos\left(\frac{2x-y}{2}\right)\cos\left(\frac{y}{2}-2a\right)dx \nonumber\\ \ &= \ \frac{1}{b^2}\left[\frac{1}{2}y\cos\left(2a-\frac{y}{2}\right)+\sin\left(\frac{y}{2}\right)\right].
\end{align}

From $x > 0$ and $y < x$, we get
\begin{align}
    V_3 \ &= \ \frac{1}{b^2}\int_{y}^{1} \cos\left(\frac{x}{2}-a\right)\cos\left(\frac{y-x}{2}-r\right)dx \nonumber\\ \ &= \ \frac{1}{b^2}\left[\frac{1}{2}(y-1)\cos\left(\frac{y}{2}\right)+\cos\left(\frac{1}{2}-2a\right)\sin\left(\frac{1-y}{2}\right)\right].
\end{align}

Since $V_1 = V_2$, we get
\begin{align}
    (f \ast f)(y) \ &= \ 2V_1+V_3 \nonumber\\ 
    \ &= \ \frac{1}{b^2}\left[\frac{1}{2}(y-1)\cos\left(\frac{y}{2}\right)+\cos\left(\frac{1}{2}-2a\right)\sin\left(\frac{1-y}{2}\right) \right. \nonumber\\
    \ & \ + \left. \frac{1}{2}y\cos\left(2a-\frac{y}{2}\right)+\sin\left(\frac{y}{2}\right)\right]
\end{align}
for $0 \le |y| \le 1$. 
We now consider the case of $y \in [0,2]$. Since $x \in [0,1]$, $x>0$ and $y-x>0$. Therefore, for $y \in [1,2]$, we get
\begin{align}
    (f \ast f)(y) \ &= \ \frac{1}{b^2}\int_{y-1}^{1} \cos\left(\frac{x}{2}-a\right)\cos\left(\frac{y-x}{2}-a\right)dx \nonumber\\ \ &= \ \frac{1}{b^2}\left[\frac{1}{2}(y-2)\sin\left(\frac{1-y}{2}\right)+\sin\left(\frac{1-y}{2}\right)\right].
\end{align}
\end{proof}

We now calculate $\phi$ using $\widehat{\phi}$. We first show that the Fourier transform of $\widehat{\phi}$ is the same as the inverse Fourier transform of $\widehat{\phi}$. Let $\mathcal{F}(g)$ denote the Fourier transform of $g$. Since $\widehat{\phi}$ is even, we get
\begin{equation}
    \mathcal{F}(\widehat{\phi}) \ = \ \int_{-\infty}^{\infty} \widehat{\phi}(y)e^{-2\pi ixy}dy \ = \ \int_{\phi}(y)e^{2\pi ixy}dy \ = \ \mathcal{F}^{-1}(\widehat{\phi}).
\end{equation}

Therefore, using Theorem \eqref{convolutionthm}, we get
\begin{equation}
    \phi(x) \ =\  \mathcal{F}^{-1}(\widehat{\phi(x)})\  = \  \mathcal{F}((f \ast f)(y)).
\end{equation}

From the Convolution Theorem (\eqref{convolutionthm}), we get
\begin{equation}
    \mathcal{F}(f \ast f(y)) \ = \ \mathcal{F}(f)F(f) \ = \ \mathcal{F}(f)^2.
\end{equation}

Therefore, we need to compute the Fourier transform of 
\begin{equation}
    f(x) \ = \ \frac{\cos\left(\frac{|x|}{2}-a\right)}{b}
\end{equation}
for $|x| \in [0,1]$. From the definition of Fourier transforms, we get
\begin{align}
    \widehat{f}(y) \ &= \ \int_{-1}^{1} \frac{\cos\left(\frac{|x|}{2}-a\right)}{b}e^{-2\pi ixy}dx \nonumber\\ \ &= \ \int_{-1}^{1} \frac{\cos\left(\frac{|x|}{2}-a\right)}{b}\cos(2\pi xy)dx + i\int_{-1}^{1} \frac{\cos\left(\frac{|x|}{2}-a\right)}{b}\sin(2\pi xy)dx.
\end{align}

Since the right integral is the integral from $-1$ to $1$ of an odd function, it evaluates to $0$. Therefore, we get
\begin{align}
    \widehat{f}(y) \ &= \ \int_{-1}^{1} \frac{\cos\left(\frac{|x|}{2}-a\right)}{b}\cos(2\pi xy)dy \nonumber\\ \ &= \ \frac{16\left(\cos(2\pi x)\sin\left(\frac{1}{2}-a\right)+\sin(a)-4\pi x\cos\left(\frac{1}{2}-a\right)\sin(2\pi x)\right)^2}{(b-16b\pi^2x^2)^2}.
\end{align}



Substituting in our values of $a$ and $b$ for the cases of SO(even) and SO(odd), we get
\begin{eqnarray}
    p_r(\mathcal{F}) \ \le \ \frac{0.8645...}{r} \ \ \ {\rm for \ G=SO(even)} \\ p_r(\mathcal{F}) \ \le \ \frac{1.1145...}{r} \ \ \ {\rm for \ G=SO(odd)}.
\end{eqnarray}

We now present the graphs of the optimal and naive test functions and their respective Fourier transforms. As shown in the plots, the two functions have values that are close to each other and thus produce similar bounds. 
\begin{center}
    \begin{tabular}{ |p{3cm}||p{3cm}|p{3cm}|p{3cm}|  }
 \hline
 \multicolumn{3}{|c|}{Comparison of Bounds for the $1$-level density for $G={\rm SO(even)}$} \\
 \hline
 Rank & Bound From naive test function & Bound From optimal test function\\
 \hline
 2 & 0.43750000 & 0.43231300\\
 4 & 0.21875000 & 0.21615700\\
 6 & 0.14583333 & 0.14410400\\
 8 & 0.10937500 & 0.10807800\\
 10 & 0.08750000 & 0.08646260\\
 12 & 0.07291670 & 0.07205220\\
 14 & 0.06250000 & 0.06175900\\
 16 & 0.05468750 & 0.05403910\\
 18 & 0.04861110 & 0.04803848\\
 20 & 0.04375000 & 0.04323130\\
 \hline
\end{tabular}
\end{center}

\begin{center}
    \begin{tabular}{ |p{3cm}||p{3cm}|p{3cm}|p{3cm}|  }
 \hline
 \multicolumn{3}{|c|}{Comparison of Bounds for the $1$-level density for $G=\rm{SO(odd)}$} \\
 \hline
 Rank & Bound From naive test function & Bound From optimal test function\\
 \hline
 1 & 1.12500000 & 1.11454000\\
 3 & 0.37500000 & 0.37151300\\
 5 & 0.22500000 & 0.22908000\\
 7 & 0.16071400 & 0.15922000\\
 9 & 0.12500000 & 0.12383838\\
 11 & 0.10227300 & 0.10132200\\
 13 & 0.08653850 & 0.08573380\\
 15 & 0.07500000 & 0.07430270\\
 17 & 0.06617650 & 0.06556120\\
 19 & 0.05921050 & 0.05866000\\
 21 & 0.05357140 & 0.05307333\\
 \hline
\end{tabular}
\end{center}

As shown in the tables, the improvements in the bounds by the optimal function are marginal, with a difference in the bounds only showing up in the hundredths digit for the non-trivial bounds. Previous work has been done by \cite{BCDMZ,FreemanMiller,ILS} to generate the optimal test functions for the $1$- and $2$-level densities. However, since the gain is so small, we choose to pursue alternate methods rather than optimizing the test functions for higher levels; we choose to look at higher levels to generate better bounds.


\section{Bounds from the $n$-Level Densities ($n$ even) with support $v=2/n$}

Previously, the largest support found was from \cite{HM} which was $v=1/(n-1)$. This has been extended to $v=2/n$ for the $n$-level density by \cite{C--}. The increased support results in a great increase in the complexity of the integrals. For support $v=1/(n-1)$, we only have a one-dimensional integral, which can be easily integrated with methods like Mathematica's NIntegrate, Simpson's Method, or Riemann approximations (we can also do directly through a contour integral, though some work is required as there is a pole of high order on the line of integration). However, the extended support introduces $n$-dimensional integrals and sums over $n$-dimensional integrals which renders these methods unusable or far too computationally intensive. 

In this section, we describe how to obtain bounds using the $n$-level density and support $v_n=2/n$. Additionally, we detail methods to lower the complexity of the calculations, including converting our $n$-dimensional integral to a $1$-dimensional integral, which allows us to use the aforementioned methods.

\begin{thm}\label{2nmainsecond}{Theorem 1.1 from \cite{C--}}
Let $n \geq 2$ and ${\rm supp}(\phi) \subset (-\frac{2}{n}, \frac{2}{n})$. Define
\begin{equation}
    \sigma_{\phi}^2 \ := \ 2\int_{-\infty}^{\infty} |y|\widehat{\phi}(y)^2 dy,
\end{equation}
\begin{multline}
    R(m, i; \phi) \ := \ 2^{m-1}(-1)^{m+1}\sum_{l=0}^{i-1} (-1)^l \binom{m}{l} \\ \left(-\frac{1}{2}\phi^m(0)  +\int_{-\infty}^{\infty} \cdots \int_{-\infty}^{\infty} \widehat{\phi}(x_2) \cdots \widehat{\phi}(x_{l+1})\right. \\ \left. \int_{-\infty}^{\infty} \phi^{m-l}(x_1)\frac{\sin(2\pi x_1(1+|x_2|+\cdots+|x_{l+1}|))}{2\pi x_1}dx_1 \cdots dx_{l+1} \right),
\end{multline},
\begin{equation}
    S(n, a, \phi) \ := \ \sum_{l=0}^{\lfloor\frac{a-1}{2}\rfloor} \frac{n!}{(n-2l)!l!}R(n-2l,a-2l,\phi)\left(\frac{\sigma_{\phi}^2}{2}\right)^l.
\end{equation}
Then
\begin{equation}\label{integralvalue2nsecond}
    \lim_{\substack{N\to\infty \\ N \text{\rm prime}}} \<
\left(D(f;\phi) - \< D(f;\phi) \>_\pm \right)^{n}\>_\pm  \ = \ (n-1)!! \sigma_{\phi}^n1_{\rm even}(n) \pm S(n, a; \phi).
\end{equation}
\end{thm}

\begin{rem}
Since the integrand is an even function in $x_1, x_2, \dots, x_{l+1}$, we can re-write $R$ as 
\begin{multline}
    R(m, i; \phi) \ := \ 2^{m-1}(-1)^{m+1}\sum_{l=0}^{i-1} (-1)^l \binom{m}{l} \\ \left(-\frac{1}{2}\phi^m(0)  + 2^{l+1}\int_{0}^{\infty} \cdots \int_{0}^{\infty} \widehat{\phi}(x_2) \cdots \widehat{\phi}(x_{l+1}) \right. \\ \left. \int_{0}^{\infty} \phi^{m-l}(x_1)\frac{\sin(2\pi x_1(1+|x_2|+\cdots+|x_l+1|))}{2\pi x_1}dx_1 \cdots dx_{l+1} \right).
\end{multline}
This decreases the computation by a factor of $1/2^{l+1}$ when computing the integral. 
\end{rem}

\begin{rem}
We can rewrite \eqref{integralvalue2nsecond} as 
\begin{equation}\label{maincenteredmomentexpansionequation}
\lim_{\substack{N\to\infty \\ N \text{\rm prime}}}\frac{1}{|\mathcal{F}_N|}\sum_{f \in \mathcal{F}_N} \left(\sum_{j} \phi\left(\gamma_{f,j}c_n\right)-\mu(\phi, \mathcal{F}_N)\right)^n \ = \ 1_{n \ \rm even}(n-1)!! \sigma_{\phi}^n \pm S(n, a; \phi)
\end{equation}
for $\phi$ an even Schwartz function with ${\rm supp}(\hphi) \subset (-\frac{2}{n}, \frac{2}{n})$ and even $n=2m$.
\end{rem}

\begin{cor}\label{maincenteredmomentexpensioninequality}
For $\phi$ an even Schwartz function with ${\rm supp}(\hphi) \subset (-\frac{1}{n-1}, \frac{1}{n-1})$ and even $n=2m$, we have 
\begin{equation}\label{maincenteredmomentexpansionequationsecond}
\lim_{\substack{N\to\infty \\ N \text{\rm prime}}}\frac{1}{|\mathcal{F}_N|}\sum_{f \in \mathcal{F}_N} \left(\sum_{j} \phi\left(\gamma_{f,j}c_n\right)-\mu(\phi, \mathcal{F})\right)^n \ = \ 1_{n \ \rm even}(n-1)!! \sigma_{\phi}^n \pm S(n, a; \phi).
\end{equation}
\end{cor}

\begin{proof}
From \eqref{maincenteredmomentexpansionequationsecond}, we get
\begin{equation}\label{maincenteredmomentexpansionequation2}
\lim_{\substack{N\to\infty \\ N \text{\rm prime}}}\frac{1}{|\mathcal{F}_N|}\sum_{f \in \mathcal{F}_N} \left(\sum_{j} \phi\left(\gamma_{f,j}c_n\right)-\mu(\phi, \mathcal{F}_N)\right)^n \ = \ 1_{n \ \rm even}(n-1)!! \sigma_{\phi}^n \pm S(n, a; \phi).
\end{equation}
We now consider the following term, $T$, and show that it equals the LHS of \eqref{maincenteredmomentexpansionequation2} 
\begin{equation}
T \ = \ \lim_{\substack{N\to\infty \\ N \text{\rm prime}}}\frac{1}{|\mathcal{F}_N|}\sum_{f \in \mathcal{F}_N} \left(\sum_{j} \phi\left(\gamma_{f,j}c_n\right) -\mu(\phi, \mathcal{F}_N) + \mu(\phi, \mathcal{F}_N) - \mu(\phi, \mathcal{F})\right)^n .
  \end{equation}

Applying the Binomial Theorem, we get
\begin{multline}
T \ = \ \lim_{\substack{N\to\infty \\ N \text{\rm prime}}}\frac{1}{|\mathcal{F}_N|}\sum_{f \in \mathcal{F}_N} \sum_{k=0}^n \binom{n}{k}\left(\sum_{j} \phi(\gamma_{f,j}c_n)-\mu(\phi, \mathcal{F}_N)\right)^k\left(\mu(\phi, \mathcal{F}_N)-\mu(\phi, \mathcal{F})\right)^{n-k}.
\end{multline}
Since $\lim_{\substack{N\to\infty \\ N \text{\rm prime}}} \left(\mu(\phi, \mathcal{F}) - \mu(\phi, \mathcal{F}_N)\right) \ = \ 0$, we would like to say each term vanishes in the limit except the $k=0$ term. Unfortunately more care is needed, as it is possible that the two terms in each product could reinforce each other. A standard application of the Cauchy-Schwarz inequality however suffices, We now have the square root of two sums. The first involves $n$-level densities to even arguments, and is $O(1)$, while the second involves our differences and thus tends to zero. Thus the only term that contributes in the limit is the first, and
\begin{multline}
    \lim_{\substack{N\to\infty \\ N \text{\rm prime}}}\frac{1}{|\mathcal{F}_N|}\sum_{f \in \mathcal{F}_N} \left(\sum_{j} \phi\left(\gamma_{f,j}c_n\right)-\mu(\phi, \mathcal{F}_N)\right)^n \\ = \lim_{\substack{N\to\infty \\ N \text{\rm prime}}}\frac{1}{|\mathcal{F}_N|}\sum_{f \in \mathcal{F}_N} \left(\sum_{j} \phi\left(\gamma_{f,j}c_n\right)  - \mu(\phi, \mathcal{F})\right)^n.\
\end{multline}

Substituting this back into our original equation, we get
\begin{equation}
    \lim_{\substack{N\to\infty \\ N \text{\rm prime}}}\frac{1}{|\mathcal{F}_N|}\sum_{f \in \mathcal{F}_N} \left(\sum_{j} \phi\left(\gamma_{f,j}c_n\right)  - \mu(\phi, \mathcal{F})\right)^n \ = \ 1_{n \ \rm even}(n-1)!! \sigma_{\phi}^n \pm S(n, a; \phi).
\end{equation}
\end{proof}

We use the following from \cite{ILS}.

\begin{thm}\label{meanphiF}
For a given test function $\phi$ where ${\rm supp}(\hphi) \subset (-1,1)$, the main term of the mean of the 1-level density of $\mathcal{F}_N$ is
\begin{equation}\label{meanphiFequation}
\mu(\phi, \mathcal{F}) \ := \ \widehat{\phi}(0)+\frac{1}{2}\int_{-1}^{1} \widehat{\phi}(y)dy.
\end{equation}
\end{thm}

\begin{thm}

For an even $n$ with $r \geq \mu(\phi, \mathcal{F})/\phi(0)$,
\begin{equation}
p_r(\mathcal{F}) \ \le \ \frac{(n-1)!!\sigma_{\phi}^{n} \pm S(n,\frac{n}{2}; \phi)}{(r\phi(0)-\mu(\phi, \mathcal{F}))^n}.
\end{equation}
 
\end{thm}

\begin{proof}
For even-level densities, the contribution to the sum in Corollary \eqref{maincenteredmomentexpensioninequality} by forms in which there are not $r$ zeroes at the central point is positive (as $n$ is even), so removing them cannot increase the sum. Therefore, from Corollary \eqref{maincenteredmomentexpensioninequality}, we have
\begin{equation}
\lim_{\substack{N\to\infty \\ N \text{\rm prime}}} \frac{1}{|\mathcal{F}_N|}\sum_{f \in \mathcal{F}_{N,r}} \left(r\phi(0) + B_f(\phi) -\mu(\phi, \mathcal{F})\right)^n \ \le \  1_{n \ \rm even}(n-1)!! \sigma_{\phi}^n \pm S(n, a; \phi).
\end{equation}

We have $(r \phi(0) + B_f(\phi) - \mu(\phi, \mathcal{F}_N))^n$; we would like to say dropping $B_f(\phi)$ cannot increase the sum, but if the first two terms are less than the third, this is not the case. By our assumption on $r$, however, we see the sum with and without $B_f(\phi)$ is positive and thus dropping it leads to an upper bound:
\begin{equation}
\lim_{\substack{N\to\infty \\ N \text{\rm prime}}} \frac{1}{|\mathcal{F}_N|}\sum_{f \in \mathcal{F}_{N,r}} \left(r\phi(0)-\mu(\phi, \mathcal{F})\right)^n \ \le \ 1_{n \ \rm even}(n-1)!! \sigma_{\phi}^n \pm S(n, a; \phi).
\end{equation}

Since the quantity $r\phi(0)-\mu(\phi, \mathcal{F}_N)$ is not dependent on $f$, we can separate it from the summation:
\begin{equation}
\left(r\phi(0)-\mu(\phi, \mathcal{F})\right)^n \lim_{\substack{N\to\infty \\ N \text{\rm prime}}}\frac{1}{|\mathcal{F}_N|}\sum_{f \in \mathcal{F}_{N,r}} 1 \  \le \ 1_{n \ \rm even}(n-1)!! \sigma_{\phi}^n \pm S(n, a; \phi).
\end{equation} 

Since $\sum_{f \in \mathcal{F}_{N, r}} 1 = |\mathcal{F}_{N, r}|$, we get
\begin{equation}
\left(r\phi(0)-\mu(\phi, \mathcal{F}_N)\right)^n \lim_{\substack{N\to\infty \\ N \text{\rm prime}}}\frac{|\mathcal{F}_{N,r}|}{|\mathcal{F}_N|} \  \le \ 1_{n \ \rm even}(n-1)!! \sigma_{\phi}^n \pm S(n, a; \phi).
\end{equation}

Using the definition for the percent that vanish to exact order $r$ from \eqref{pFN}, we get
\begin{equation}
\lim_{\substack{N\to\infty \\ N \text{\rm prime}}} p_r(\mathcal{F}_N) \ \le \ \frac{(n-1)!!\sigma_{\phi}^{n} \pm R_{n}(\phi)}{(r\phi(0)-\mu(\phi, \mathcal{F}_N))^n}.
\end{equation}

Taking the limit as $N \to \infty$, we get
\begin{equation}
p_r(\mathcal{F}) \ \le \ \frac{1_{n \ \rm even}(n-1)!! \sigma_{\phi}^n \pm S(n, a; \phi)}{(r\phi(0)-\mu(\phi, \mathcal{F}_N))^n}.
\end{equation}
\end{proof}

For quicker computation, we can convert the multi-dimensional $R$ integral to a one-dimensional integral. 
\begin{lem}
Let 
\begin{equation}
    I_1(x, v) \ := \ \frac{2i \pi vx - e^{2i \pi vx} + 1}{2\pi^2 v^2 x^2} \ \ \ , \ \ \  I_2(x,v) \ := \ \frac{1-e^{-2i \pi vx}-2i\pi vx}{2\pi^2v^2x^2} = I_1(x,v). 
\end{equation}
\begin{multline}
    R(m, i; \phi) \ = \ 2^{m-1}(-1)^{m+1}\sum_{l=0}^{i-1}(-1)^l\binom{m}{l} \\ \left(-\frac{1}{2}\phi^m(0) + \int_{-\infty}^{\infty} \phi^{m-l}(x)\frac{I_1(x)^le^{2\pi ix}-I_2(x)^le^{-2\pi ix}}{4\pi ix}\right).
\end{multline}
\end{lem}

\begin{proof}
Let
\begin{multline}
    T \ := \ \int_{-\infty}^{\infty} \cdots \int_{-\infty}^{\infty} \widehat{\phi}(x_2) \cdots \widehat{\phi}(x_{l+1}) \\ \int_{-\infty}^{\infty} \phi^{m-l}(x_1) \frac{\sin(2\pi x_1(1+|x_2|+\cdots+|x_{l+1}|))}{2\pi x_1}dx_1. \cdots dx_{l+1}.
\end{multline}
Then, we get
\begin{equation}
    R(m, i; \phi) \ = \ 2^{m-1}(-1)^{m+1}\sum_{l=0}^{i-1}(-1)^l\binom{m}{l}\left(-\frac{1}{2}\phi^m(0) + T\right).
\end{equation}
We reduce $T$ to a 1-dimensional integral. To do so, we first show that the conditions for Fubini's Theorem hold which allows us to switch the order of integration. Therefore, we show that integral of the absolute value converges. Let 
\begin{multline}
    T'\ := \ \int_{-\infty}^{\infty} \cdots \int_{-\infty}^{\infty} \widehat{\phi}(x_2) \cdots \widehat{\phi}(x_{l+1}) \\ \int_{-\infty}^{\infty} \phi^{m-l}(x_1)\left|\frac{\sin(2\pi x_1(1+|x_2|+\cdots+|x_{l+1}|))}{2\pi x_1}\right|dx_1 \cdots dx_{l+1}.
\end{multline}
Since $\phi$ and $\widehat{\phi}$ are non-negative and the absolute value is non-negative, the value inside the integrand is non-negative. Multiplying by $(1+|x_2|+\cdots+|x_{l+1})/(1+|x_2|+\cdots+|x_{l+1}|)$, we get
\begin{multline}
    T' \ = \ \int_{-\infty}^{\infty} \cdots \int_{-\infty}^{\infty} \widehat{\phi}(x_2) \cdots \widehat{\phi}(x_{l+1})(1+|x_2|+\cdots+|x_{l+1}|)\\ \int_{-\infty}^{\infty} \phi^{m-l}(x_1)\left|\frac{\sin(2\pi x_1(1+|x_2|+\cdots+|x_{l+1}|))}{2\pi x_1(1+|x_2|+\cdots+|x_{l+1}|)}\right|dx_1 \cdots dx_{l+1}.
\end{multline}
Since $|\sin(u)/u| \le 1$ for all $u$, we get
\begin{equation}
    T' \ \le \ \int_{-\infty}^{\infty} \cdots \int_{-\infty}^{\infty} \widehat{\phi}(x_2) \cdots \widehat{\phi}(x_{l+1})(1+|x_2|+\cdots+|x_{l+1}|)\int_{-\infty}^{\infty} \phi^{m-l}(x_1)dx_1 \cdots dx_{l+1}.
\end{equation}
Since $\widehat{\phi}$ decays faster than any polynomial, the integral converges so Fubini's Theorem holds. Therefore, we can apply Fubini's Theorem to get
\begin{multline}\label{fubini}
    T \ = \ \int_{-\infty}^{\infty} \cdots \int_{-\infty}^{\infty} \widehat{\phi}(x_2) \cdots \widehat{\phi}(x_{l+1}) \\  \int_{-\infty}^{\infty}  \phi^{m-l}(x_1)\frac{\sin(2\pi x_1(1+|x_2|+\cdots+|x_{l+1}|))}{2\pi x_1}dx_1 \cdots dx_{l+1}. 
\end{multline}
Converting sine to exponential form, we get
\begin{multline}
    T \ = \ \int_{-\infty}^{\infty} \cdots \int_{-\infty}^{\infty} \widehat{\phi}(x_2) \cdots \widehat{\phi}(x_{l+1}) \\ \int_{-\infty}^{\infty} \phi^{m-l}(x+1)\frac{e^{2\pi ix_1} \cdots e^{2\pi ix_1|x_{l+1}|}-e^{-2\pi ix_1} \cdots e^{-2\pi ix_1|x_{l+1}|}}{4\pi ix_1} dx_2 \cdots dx_{l+1} dx_1. 
\end{multline}
    
To evaluate this integral, we find
\begin{align}\label{integralone}
    \int_{-\infty}^{\infty} \widehat{\phi}(y)e^{2\pi ix_1|y|} \ &= \ 2\int_{0}^{v_n} \widehat{\phi}(y)e^{2\pi ix_1y} dy \nonumber\\ \ &= \ \frac{2i\pi vx-e^{2i\pi vx}+1}{2\pi^2 v^2 x^2} = I_1(x, v). 
 \end{align}

Additionally, we find
\begin{align}\label{integral2}
    \int_{-\infty}^{\infty} \widehat{\phi}(y)e^{-2\pi x_1|y|} \ &= \ 2\int_{0}^{v_n} \widehat{\phi}(y)e^{-2\pi ix_1|y|} dy \nonumber\\ \ &= \ \frac{1-e^{-2i\pi vx - 2i\pi vx}}{2\pi^2v^2x^2} = \ I_2(x, v).
\end{align}
Plugging these values into \eqref{fubini}, we get
\begin{align}\label{finalTvalue}
    T \ &= \ \int_{-\infty}^{\infty} I_1(x_1)^{l}\phi^{m-l}(x_1)\frac{e^{2\pi ix_1}}{4\pi ix_1} - \int_{-\infty}^{\infty} I_2(x_1)^{l}\phi^{m-l}(x_1)\frac{e^{-2\pi ix_1}}{4\pi ix_1} \nonumber\\ \ &= \ \int_{-\infty}^{\infty} \phi^{m-l}(x)\frac{I_1(x)^le^{2\pi ix}-I_2(x)^le^{-2\pi ix}}{4\pi ix}.
\end{align}
Therefore, plugging in our value of $T$ from \eqref{finalTvalue}, we get
\begin{multline}
    R(m, i; \phi) \ = \ 2^{m-1}(-1)^{m+1}\sum_{l=0}^{i-1}(-1)^l\binom{m}{l} \\ \left(-\frac{1}{2}\phi^m(0) + \int_{-\infty}^{\infty} \phi^{m-l}(x)\frac{I_1(x)^le^{2\pi ix}-I_2(x)^le^{-2\pi ix}}{4\pi ix}\right).
\end{multline}
\end{proof}

\begin{rem}
From the equations of $I_1(x, v)$ and $I_2(x, v)$, we see that $I_1(x, v)$ and $I_2(x, v)$ are complex conjugates. Therefore, when calculating the difference of $I_1(x, v)$ and $I_2(x, v)$, we only need to calculate $2\Re(I_1(x, v))$ which reduces computation.
\end{rem}

\normalfont

We now want to calculate the values of $S$ for small cases to find bounds using support $2/n$. We first consider the cases of $n=2$ and $n=4$ with the naive test function. 

\section{Plots of Bounds for support $v=2/n$}
We display bounds found using support $v=2/n$. We first present bounds for the naive test function.

\begin{figure}[h]
\begin{center}
\scalebox{1}{\includegraphics{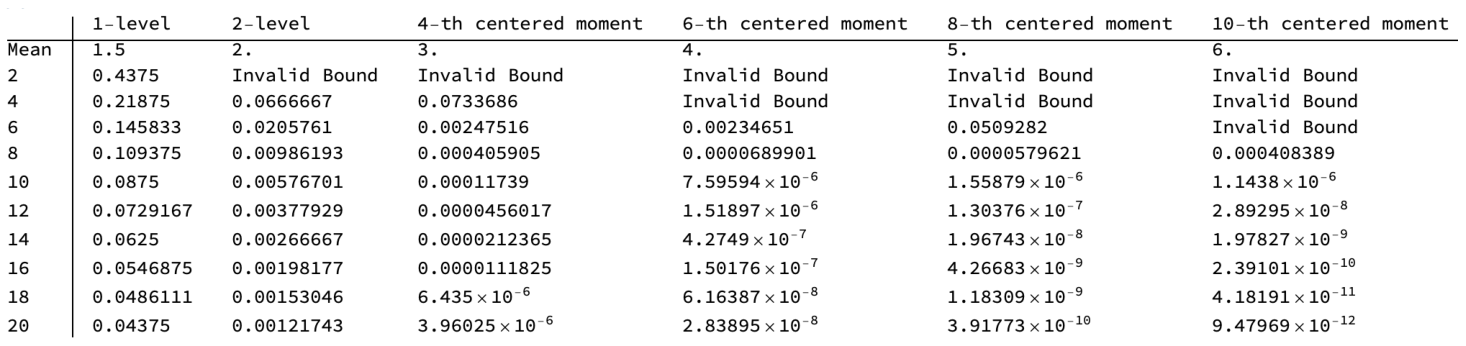}}\caption{\label{fig:sheet5} Approximate Bounds for the Percent of Vanishing to exact order $r$ for the case G=SO(even) with support $v=2$ for the $1$-level and $v=2/n$ for the $n$-level with $n$ going from $1$ to $10$ and $r$ from $2$ through $20$ obtained using the naive test function.}
\end{center}\end{figure}
\begin{figure}[h]
\begin{center}
\scalebox{1}{\includegraphics{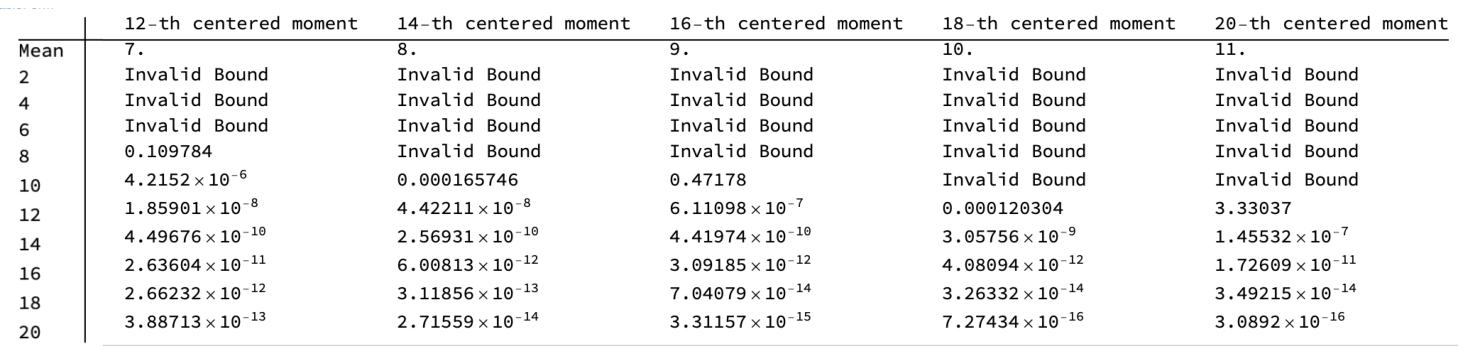}}\caption{\label{fig:sheet6} Approximate Bounds for the Percent of Vanishing to exact order $r$ for the case G=SO(even) with support $v=2$ for the $1$-level and $v=2/n$ with $n$ going from 12 through 20 and $r$ going from $2$ through $20$ obtained using the naive test function.}
\end{center}\end{figure}
\begin{figure}[h]
\begin{center}
\scalebox{1}{\includegraphics{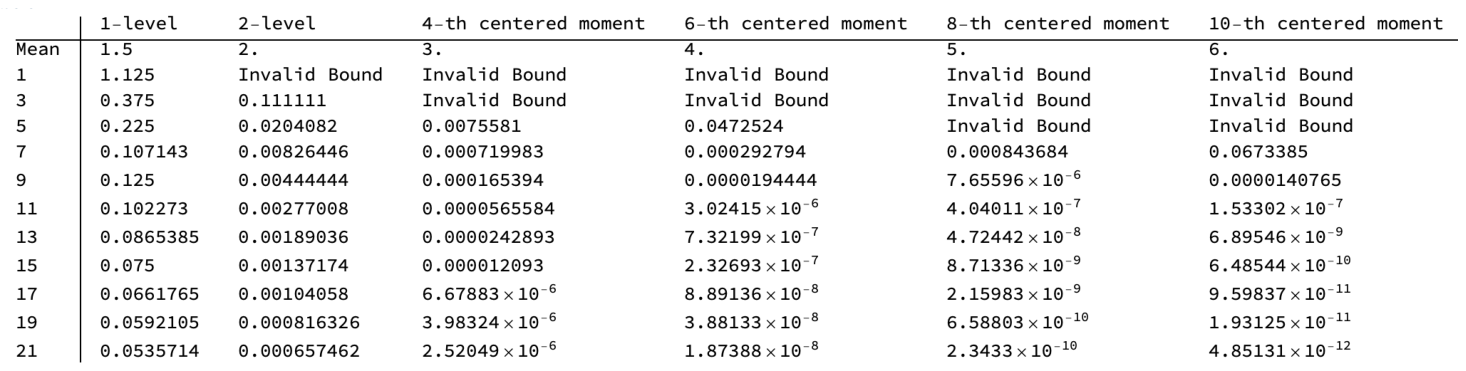}}\caption{\label{fig:sheet7} Approximate Bounds for the Percent of Vanishing to exact order $r$ for the case G=SO(odd) with support $v=2$ for the $1$-level and $v=2/n$ with $n$ going from $1$ through $10$ and $r$ going from $1$ through $21$ obtained using the naive test function.}
\end{center}\end{figure}
\begin{figure}[h]
\begin{center}
\scalebox{1}{\includegraphics{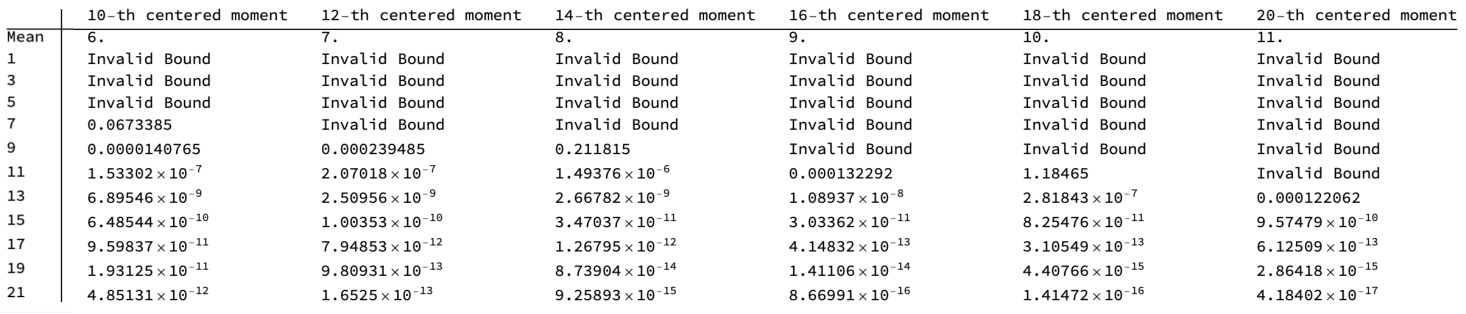}}\caption{\label{fig:sheet8} Approximate Bounds for the Percent of Vanishing to exact order $r$ for the case G=SO(odd) with support $v=2$ for the $1$-level and $v=2/n$ with $n$ going from $12$ through $20$ and $r$ going from $1$ through $21$ obtained using the naive test function.}
\end{center}\end{figure}
We now show our bounds using the optimal test function for the 1-level density. 
\begin{figure}[h]
\begin{center}
\scalebox{1}{\includegraphics{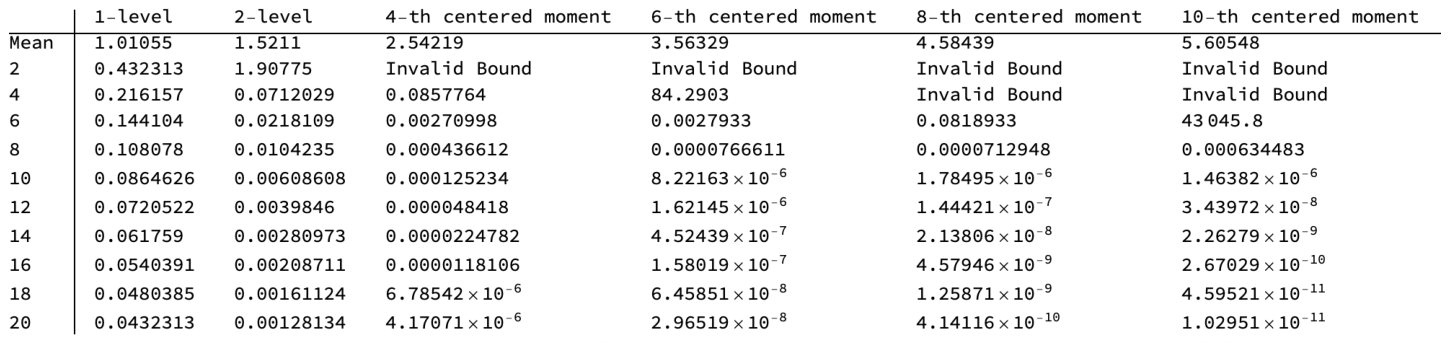}}\caption{\label{fig:sheet9} Approximate Bounds for the Percent of Vanishing to exact order $r$ for the case G=SO(even) with support $v=2$ for the $1$-level and $v=2/n$ for the n-level with $n$ going from $1$ to $10$ and $r$ from $2$ through $20$ obtained using the optimal test function.}
\end{center}\end{figure}
\begin{figure}[h]
\begin{center}
\scalebox{1}{\includegraphics{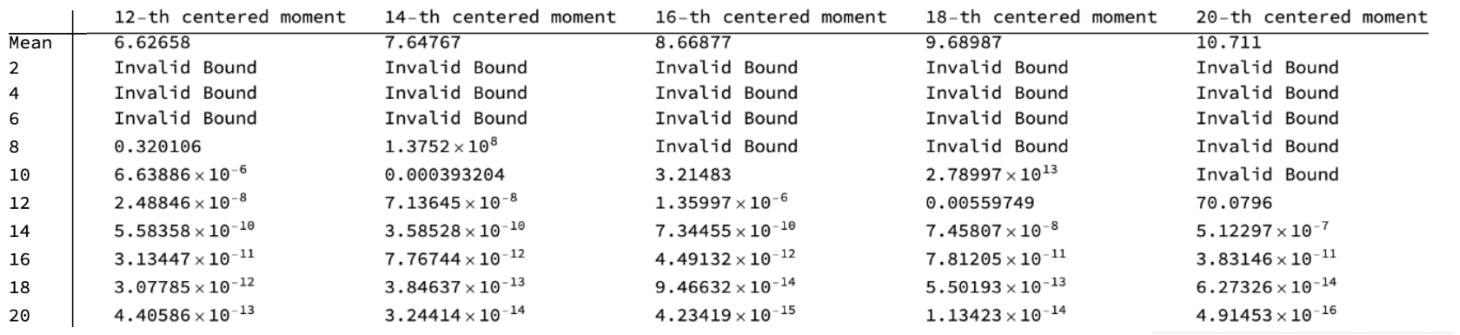}}\caption{\label{fig:sheet10} Approximate Bounds for the Percent of Vanishing to exact order $r$ for the case G=SO(even) with support $v=2$ for the $1$-level and $v=2/n$ with $n$ going from 12 through 20 and $r$ going from $2$ through $20$ obtained using the optimal test function.}
\end{center}\end{figure}
\begin{figure}[h]
\begin{center}
\scalebox{1}{\includegraphics{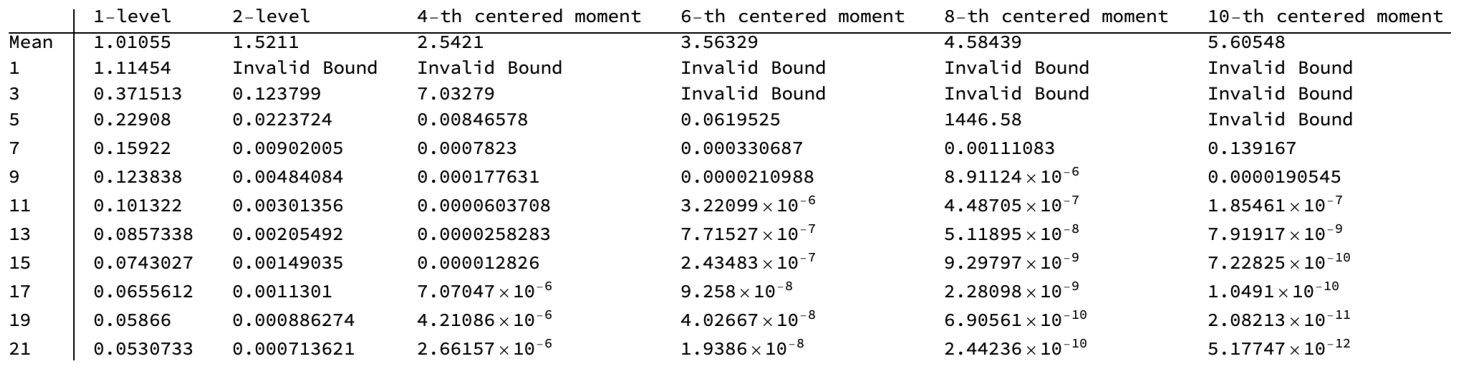}}\caption{\label{fig:sheet11} Approximate Bounds for the Percent of Vanishing to exact order $r$ for the case G=SO(odd) with support $v=2$ for the $1$-level and $v=2/n$ with $n$ going from $1$ through $10$ and $r$ going from $1$ through $21$ obtained using the optimal test function.}
\end{center}\end{figure}
\begin{figure}[ht]
\begin{center}
\scalebox{1}{\includegraphics{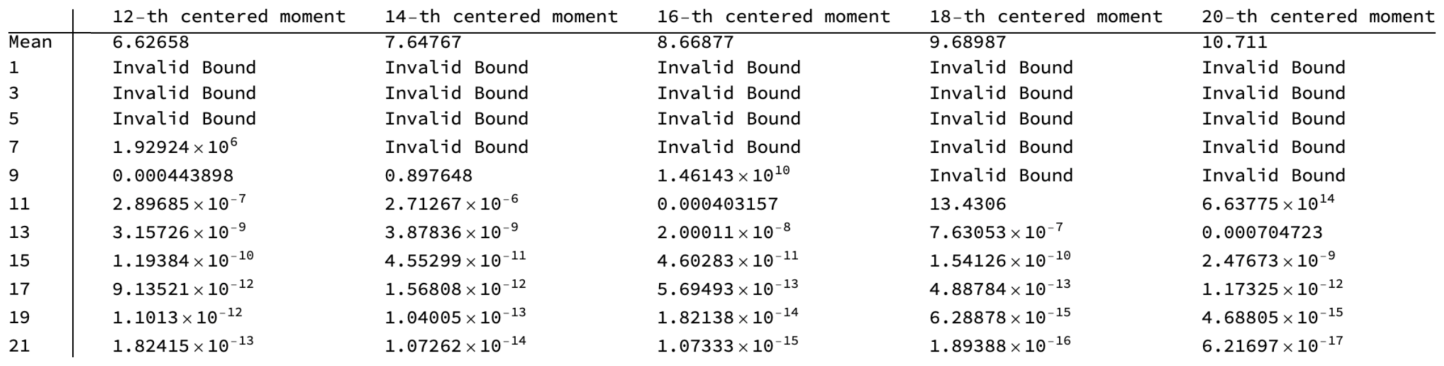}}\caption{\label{fig:sheet12} Approximate Bounds for the Percent of Vanishing to exact order $r$ for the case G=SO(odd) with support $v=2$ for the $1$-level and $v=2/n$ with $n$ going from $12$ through $20$ and $r$ going from $1$ through $21$ obtained using the optimal test function.}
\end{center}\end{figure}
As shown in the Figure \eqref{fig:sheet5}, the 1-level optimal test function does not produce better results than the naive test function for higher levels. While the bounds for the 1-level optimal are better for the 1-level, interestingly the naive test function outperforms it for higher levels and ranks.

See the table in \S\ref{sec:mainresults} for a summary of the bounds attainable using the best choice of level and the naive test function for each rank from $2$ through $20$.


\appendix



\ \\

\thispagestyle{empty}


\clearpage
\begin{thebibliography}{99009999}

\bibitem[BFMT-B]{BFMT-B}
O. Barrett, F. W. K. Firk, S. J. Miller and C. Turnage-butterbaugh, \emph{From Quantum Systems to $L$-Functions: Pair Correlation Statistics and Beyond}, in Open Problems in Mathematics (editors John Nash  Jr. and Michael Th. Rassias), Springer-Verlag, 2016, pages 123--171.

\bibitem[BSD1]{BSD1}
\newblock B. Birch and H. Swinnerton-Dyer, \emph{Notes on elliptic
curves. I}, J. reine angew. Math. \textbf{212} (1963), 7--25.

\bibitem[BSD2]{BSD2}
\newblock B. Birch and H. Swinnerton-Dyer, \emph{Notes on elliptic
curves. II}, J. reine angew. Math. \textbf{218} (1965), 79--108.

\bibitem[BCDMZ]{BCDMZ}
\newblock E. Boldyriew, F. Chen, C. Devlin VI, S. J. Miller, and J. Zhao, \emph{Determining optimal test functions for 2-level densities}, to appear in Research in Number Theory.

\bibitem[C--]{C--}
\newblock Peter Cohen, Justine Dell, Oscar E. Gonzalez, Geoffrey Iyer, Simran Khunger, Chung-Hang Kwan, Steven J. Miller, Alexander Shashkov, Alicia
Smith Reina, Carsten Sprunger, Nicholas Triantafillou, Nhi Truong, Roger Van Peski, Stephen Willis, and Yingzi Yang, \emph{Extending Support for the Centered Moments of the Low-Lying Zeroes Of Cuspidal Newforms}, preprint (2022), \bburl{https://arxiv.org/pdf/2208.02625}. 

\bibitem[FM]{FreemanMiller}
J. Freeman and S. J. Miller, \emph{\newblock Determining optimal test functions for bounding the average rank in
  families of $l$-functions}, in SCHOLAR -- a Scientific Celebration Highlighting Open Lines of Arithmetic Research, Conference in Honour of M. Ram Murty's Mathematical Legacy on his 60th Birthday (A. C. Cojocaru, C. David and F. Pappaardi, editors), Contemporary Mathematics \textbf{655}, AMS and CRM, 2015.

\bibitem[Ha]{Ha}
B. Hayes, \emph{The spectrum of Riemannium}, American Scientist
\textbf{91} (2003), no. 4, 296--300.

\bibitem[Hej]{Hej}
D. Hejhal, \emph{On the triple correlation of zeros of the zeta
function}, Internat. Math. Res. Notices (1994), no. 7, 294--302.

\bibitem[HM]{HM}
\newblock C. Hughes and S. J. Miller, \emph{Low-lying zeros of $L$-functions
with orthogonal symmetry}, Duke Math. J.
\textbf{136} (2007), no. 1, 115--172.

\bibitem[IK]{IK}
H. Iwaniec and E. Kowalski, \emph{Analytic Number Theory}, AMS
Colloquium Publications, Vol. 53, AMS, Providence, RI, 2004.

\bibitem[ILS]{ILS}
H. Iwaniec, W. Luo, and P. Sarnak, \emph{Low lying zeros of families of L-functions}, Inst. Hautes \'Etudes Sci. Publ. Math. \textbf{91} (2000), 55-131.

\bibitem[KS1]{KatzSarnak}
N. Katz and P. Sarnak, \emph{Random Matrices, Frobenius Eigenvalues and Monodromy}, AMS Colloquium Publications, Vol. 45, AMS,
Providence, RI, 1999.

\bibitem[KS2]{KatzSarnak2}
N. Katz and P. Sarnak, \emph{Zeros of zeta functions and symmetries}, Bull. AMS \textbf{36} (1999), 1--26.

\bibitem[LiM]{LiM}
J. Li and S. J. Miller, \emph{Bounding Vanishing at the Central Point of Cuspidal Newforms}, to appear in the Journal of Number Theory.

\bibitem[Mon]{Mon} H. Montgomery, \emph{The pair correlation of zeros of
  the zeta function}. Pages 181--193 in \emph{Analytic Number Theory},
  Proceedings of Symposia in Pure Mathematics, vol. 24, AMS,
  Providence, RI, 1973.

\bibitem[Od1]{Od1} A. Odlyzko, \emph{On the distribution of spacings
  between zeros of the zeta function}, Math. Comp. \textbf{48} (1987),
  no. $177$,  273--308.

\bibitem[Od2]{Od2} A. Odlyzko, \emph{The $10^{22}$-nd zero of the
  Riemann zeta function}. Pages  139--144 in \emph{Proceedings of the
  Conference on Dynamical, Spectral and Arithmetic Zeta Functions},
  ed. M. van Frankenhuysen and M. L. Lapidus, Contemporary Mathematics
  Series,AMS, Providence, RI, $2001$.

\bibitem[RS]{RudnickSarnak}
Z. Rudnick and P. Sarnak, \emph{Zeros of principal $L$-functions and random matrix theory}, Duke J. of Math. \textbf{81} (1996),
269--322.

\bibitem[SS]{SS}
E. Stein and R. Shakarchi, \emph{Complex Analysis}, Princeton University, Princeton, NJ, 2003.

\end{thebibliography}
\end{document}